\documentclass{amsart}
\usepackage{amssymb,amsfonts,amsmath,amsthm,amscd,stmaryrd,dsfont,esint,hyperref,upgreek,xcolor}
\usepackage[title]{appendix}
\usepackage[mathcal]{euscript}

\usepackage{fullpage}
\theoremstyle{plain}
\newtheorem{thm}{Theorem}
\newtheorem{cor}{Corollary}
\newtheorem{lem}[cor]{Lemma}

\theoremstyle{definition}
\newtheorem{defn}[cor]{Definition}

\usepackage{biblatex}
{}
{}
\DeclareMathOperator*{\argmax}{arg\,max}
\DeclareMathOperator*{\argmin}{arg\,min}

\DeclareMathOperator{\Lip}{Lip}

\newcommand{\R}{\mathbb{R}}
\newcommand{\C}{\mathbb{C}}
\newcommand{\Z}{\mathbb{Z}}
\newcommand{\N}{\mathbb{N}}
\newcommand{\Q}{\mathbb{Q}}

\renewcommand{\tilde}{\widetilde}

\title{A near-optimal rate of periodic homogenization for convex Hamilton-Jacobi equations}
\author{William Cooperman}
\begin{document}
\begin{abstract}
    We consider a Hamilton-Jacobi equation where the Hamiltonian is periodic in space and coercive and convex in momentum. Combining the representation formula from optimal control theory and a theorem of Alexander, originally proved in the context of first-passage percolation, we find a rate of homogenization which is within a log-factor of optimal and holds in all dimensions.
\end{abstract}
\maketitle
\section{Introduction}
Let the Hamiltonian $H \colon \R^d \times \R^d \to \R$ be continuous, $\Z^d$-periodic in the first variable, $x$, and coercive in the second variable, $p$. We assume that the coercivity is uniform in $x$; that is, \[ \liminf_{|p| \to \infty} \inf_{x \in \R^d} H(x, p) = +\infty. \]
Let $u_0 \colon \R^d \to \R$ be continuous. Our goal is to study, as $\varepsilon \to 0^+$, the behavior of the unique viscosity solution $u^\varepsilon \colon \R_{\geq 0} \times \R^d \to \R$ to the initial-value problem
\begin{equation}\label{eps-cauchy}
    \begin{cases}
        D_t u^\varepsilon(t, x) + H(\frac{x}{\varepsilon}, D_x u^\varepsilon(t, x)) = 0 \qquad &\text{in $\R_{>0} \times \R^d$}\\
        u^\varepsilon(0, x) = u_0(x) \qquad &\text{in $\R^d$}.
    \end{cases}
\end{equation}
Lions\textendash{}Papanicolaou\textendash{}Varadhan~\cite{LPV} proved that $u^\varepsilon \to \overline{u}$ locally uniformly as $\varepsilon \to 0^+$, where $\overline{u} \colon \R_{\geq 0} \times \R^d \to \R$ is the solution to the effective problem
\begin{equation}\label{macro-cauchy}
    \begin{cases}
        D_t \overline{u}(t, x) + \overline{H}(D_x \overline{u}(t, x)) = 0 \qquad &\text{in $\R_{>0} \times \R^d$}\\
        \overline{u}(0, x) = u_0(x) \qquad &\text{in $\R^d$}.
    \end{cases}
\end{equation}
Here, $\overline{H} \colon \R^d \to \R$ is called the effective Hamiltonian; we define $\overline{H}(p)$ as the unique constant such that the cell problem
\begin{equation}\label{cell-problem}
    H(x, p + D_x v_p) = \overline{H}(p)
\end{equation}
has some $\Z^d$-periodic continuous viscosity solution $v_p \colon \R^d \to \R$, called a corrector.

Our main result is the following rate of convergence, under additional assumptions on $H$ and $u_0$.
\begin{thm}
    If $H$ is convex in $p$ and $u_0$ is Lipschitz, then there is a constant $C = C(H, \Lip(u_0)) > 0$ such that, for all $t > 0$ and $x \in \R^d$, \[ |u^\varepsilon(t, x) - \overline{u}(t, x)| \leq C\varepsilon\log(C + t\varepsilon^{-1}). \]
\end{thm}
Additionally, in the case of dimension $d=2$, we provide a new proof of a result of Mitake\textendash{}Tran\textendash{}Yu.
\begin{thm}[Mitake\textendash{}Tran\textendash{}Yu~\cite{MTY}]
    If $d=2$, $H$ is convex in $p$, and $u_0$ is Lipschitz, then there is a constant $C = C(H, \Lip(u_0)) > 0$ such that, for all $t > 0$ and $x \in \R^d$, \[ |u^\varepsilon(t, x) - \overline{u}(t, x)| \leq C\varepsilon. \]
\end{thm}
The proofs exploit the control formulation of the initial value problem~(\ref{eps-cauchy}), which reduces homogenization to a question about convergence of a subadditive function. In both the $d=2$ and $d \geq 3$ case, results of Alexander~\cite{Alexander97}~\cite{Alexander90} apply to quantify the convergence.

Two months after we posted this article, Hung Tran and Yifeng Yu pointed out that, by replacing Step 1 in our proof of Lemma~\ref{approx-additivity} with Lemma 2 of Burago~\cite{Burago}, one obtains the optimal $O(\varepsilon)$ rate in all dimensions. In fact, this key lemma is exactly the Hobby\textendash{}Rice theorem~\cite{HobbyRice}, proved in 1965.

\section{Prior work}
After Lions\textendash{}Papanicolaou\textendash{}Varadhan proved qualitative homogenization, there have been two main quantitative results. Under the assumptions that $u_0$ is Lipschitz and $H$ is locally Lipschitz, Capuzzo-Dolcetta\textendash{}Ishii~\cite{CDI} proved a rate of $O(\varepsilon^{1/3})$, using the perturbed test function method with approximate correctors. Under the additional assumption that $H$ is convex in $p$, Mitake\textendash{}Tran\textendash{}Yu~\cite{MTY} proved a rate of $O(\varepsilon)$ in dimension $d=2$ and a rate of $O(\varepsilon^{1/2})$ in dimensions $d \geq 3$ using weak KAM theory.

From the definition~(\ref{cell-problem}) of $\overline{H}$, we can heuristically hope for the expansion \[ u^\varepsilon(t, x) \approx \overline{u}(t, x) + \varepsilon v_{D_x \overline{u}(t, x)}(\varepsilon^{-1}x), \] which suggests a rate of $O(\varepsilon)$. However, the correctors are not unique, $u$ is not $C^1$ but only Lipschitz, and a continuous selection $p \mapsto v_p$ of correctors (let alone a Lipschitz selection) does not exist in general (see section 5 of~\cite{MTY} for an example). The assumptions on the initial data and the Hamiltonian help by giving additional structure to the problem, in the form of the control formulation.

\section{Subadditive convergence}
We begin by presenting a result of Alexander. In this section, we let $\Omega \subseteq \R^N$ denote an open convex cone. First, we make a few definitions.
\begin{defn}
    A function $f \colon \Omega \cap \Z^N \to \R_{\geq 0}$ has \textit{approximate geodesics} if there is a constant $K > 0$ such that, for every $x \in \Omega \cap \Z^N$, there are $x_0, x_1, \dots, x_n \in \Omega \cap \Z^N$ with $x_0 = 0$, $x_n = x$, $x_{i+1}-x_i \in \Omega$, $|x_{i+1}-x_i| \leq K$, and \[ |f(x_k-x_i) - f(x_k-x_j) - f(x_j-x_i)| \leq K \] for all $i \leq j \leq k$.
\end{defn}
\begin{defn}
    A function $f \colon \Omega \cap \Z^N \to \R_{\geq 0}$ is subadditive if $f(x+y) \leq f(x) + f(y)$ for all $x, y \in \Omega \cap \Z^N$.
\end{defn}
\begin{defn}
    A function $f \colon \Omega \cap \Z^N \to \R_{\geq 0}$ has linear growth if there is a constant $K \geq 1$ such that $K^{-1}|x| - K \leq f(x) \leq K|x| + K$ for all $x \in \Omega \cap \Z^N$.
\end{defn}
\begin{thm}[Alexander~\cite{Alexander97}]
    If $f \colon \Omega \cap \Z^N \to \R_{\geq 0}$ is subadditive, has linear growth, and has approximate geodesics, then there is a constant $C > 0$ such that, for all $x \in \Omega \cap \Z^N$, \[ |f(x) - \lim_{n \to \infty} n^{-1}f(nx)| \leq C\log(C + |x|). \]
\end{thm}
\begin{proof}
    Without loss of generality, we assume $K \geq 2$. Define $\overline{f} \colon \Omega \cap \Q^N \to \R$ by \[ \overline{f}(x) := \lim_{n \to \infty} n^{-1}f([nx]), \] where $[\cdot]$ denotes coordinate-wise rounding to integers. Then $\overline{f}$ is also subadditive with linear growth. From the scaling, it is immediate that $t\overline{f}(x) = \overline{f}(tx)$ for all $t \geq 0$. From subadditivity of $f$, we see that $\overline{f} \leq f$.

    For the rest of the argument, we let $C > 1 > c > 0$ be constants which depend only on $K$ and $N$ and may differ from line to line.

    For each $x \in \Omega \cap \Q^N$, define $\overline{f}_x$ to be a supporting affine functional of $\overline{f}$ at $x$, chosen consistently so that $\overline{f}_{tx} = \overline{f}_x$ for all $t > 0$. We think of $\overline{f}_x(v)$ as the amount of progress that a step $v$ makes in the direction $x$. Given $x \in \Omega \cap \Q^N$, define the set of ``good'' increments \[ Q_x = \left\{v \in \Omega \cap \Z^N \mid f(v) - 5K^2 \leq \overline{f}_x(v) \leq \overline{f}(x)\right\}. \] We think of $f(v)-\overline{f}_x(v)$ as the amount of inefficiency in the increment $v$ on a path toward $x$, so a good increment is one which has inefficiency at most $5K^2$. The second part of the inequality means that good increments don't ``overshoot'' in the direction of $x$, which implies (from linear growth) that good increments have length at most $C|x|$.

    \textit{Step 1.} We show that if $x \in \Omega \cap \Q^N$ with $|x| \geq C$, then there is $\alpha \in [c, 1]$ such that $\alpha x$ lies in the convex hull of $Q_x$. Let $n \in \N$ be large enough so that \[ |n^{-1}f(nx) - \overline{f}(x)| \leq 1. \] Let $x_0, x_1, \dots, x_m$ be an approximate geodesic for $nx$. We iteratively define a subsequence $y_k = x_{j_k}$ by letting $j_0 = 0$ and, as long as $j_k < m$, we define $j_{k+1} \in [j_k+1, \dots, m]$ to be maximal such that $y_{k+1}-y_k \in Q_x$. By linear growth of $f$ and $\overline{f}$, we have \[ \overline{f}_x(x_{j_k+1}-x_{j_k}) - 5K^2 \leq K^2 + K - 5K^2 \leq f(x_{j_k+1}-x_{j_k}) \leq K^2 + K \leq K^{-1}C - K \leq \overline{f}(x), \] as long as $C$ was chosen large enough, so $j_k+1$ is admissible and therefore the subsequence exists, and we let $p \in \N$ be the index where $j_p = m$. If $k$ is such that $j_{k+1} < m$ and $\overline{f}_x(x_{j_{k+1}+1}-x_{j_k}) > \overline{f}(x)$, then the fact that $|x| \geq C$, linear growth, and the approximate geodesic property yields \[ f(y_{k+1}-y_k) \geq (K^{-1}|x| - K) - (K^2+2K). \] Choosing $C$ large enough and summing over $k$ (using the approximate geodesic property again) shows that there are $O(n)$ many such $k$.

    On the other hand, let $\ell$ be the number of $k$ such that $j_{k+1} < m$ and \[ f(x_{j_{k+1}+1}-x_{j_k}) - 5K^2 > \overline{f}_x(x_{j_{k+1}+1}-x_{j_k}). \] For such $k$, we have \[ \overline{f}_x(y_{k+1}-y_k) \leq f(y_{k+1}-y_k) - 5K^2 + 2(K^2+K) \leq f(y_{k+1}-y_k) - K^2. \] Linearity of $\overline{f}_x$  and the approximate geodesic property shows that \[ \overline{f}_x(nx) = n\overline{f}(x) \leq f(nx) + pK - \ell K^2, \] so the choice of $n$ implies that $\ell K^2 - pK \leq n$ and therefore $\ell \leq \frac14 n + \frac12 p$. All together, we have shown that $p \leq Cn$. We conclude this step by noting that \[ x = \frac{1}{n}\sum_{k=1}^p (y_k-y_{k-1}), \] and $n \leq p \leq Cn$ (the first part of the inequality follows from applying $\overline{f}_x$ to both sides of the equation).

    \textit{Step 2.} We show that if $x \in \Omega \cap \Q^N$, $|x| \geq C$, $t \geq 1$, and $tx \in \Z^N$, then there is a $z \in \Omega \cap \Z^N$ with \[ f(tx) - \overline{f}(tx) \leq f(z) - \overline{f}(z) + Ct. \] Using the previous step, write $tx = z + \sum_{k=1}^m v_k$, where $|z| \leq C|x|$, $\overline{f}(z) \leq \overline{f}_x(z) + C$, $v_k \in Q_x$, and $m \leq Ct$. Indeed, for some $\alpha \in [c, 1]$ we first write \[ \alpha x = \sum_{i=1}^{N+1} p_i v_i, \] where $v_i \in Q_x$ and $p_i \geq 0$, $\sum_i p_i = 1$. Note that the sum only requires $N+1$ terms by Caratheodory's theorem on convex hulls, since we are working in $\R^N$. To decompose $tx$, we write \[ tx = \sum_{i=1}^{N+1}(t\alpha^{-1}p_i - \lfloor t\alpha^{-1}p_i \rfloor)v_i + \sum_{i=1}^{N+1}\lfloor t\alpha^{-1}p_i \rfloor v_i =: z + (tx-z), \] so $z$ satisfies the required properties. By subadditivity of $f$ and linearity of $\overline{f}_x$,
    \begin{align*}
        f(tx) &\leq f(z) + \sum_{k=1}^m f(v_k)\\
        &\leq f(z) + \sum_{k=1}^m \left(\overline{f}_x(v_k) + 5K^2\right)\\
        &= f(z) + \overline{f}_x(tx - z) + 5CK^2t\\
        &\leq f(z) + \overline{f}_x(tx-z) + Ct.
    \end{align*}
    Finally, we write $\overline{f}(tx) = \overline{f}_x(z) + \overline{f}_x(tx-z)$ and subtract from both sides of the inequality above to get \[ f(tx) - \overline{f}(tx) \leq f(z) - \overline{f}(z) + Ct, \] where we used the fact that $\overline{f}(z) \leq \overline{f_x}(z) + C$.

    \textit{Step 3.} For some large $M > 1$, the previous step yields \[ \sup_{|x| \leq M^{k+1}C} f(x)-\overline{f}(x)  \leq \sup_{|x| \leq M^k C} f(x)-\overline{f}(x) + CM. \] We conclude by induction on $k$.
\end{proof}
\section{Homogenization via the metric problem}
Let $C > 1 > c > 0$ denote constants which depend on $H$ and $\Lip(u_0)$ and may differ from line to line. If $a \in \R$, then replacing $H$ by $H-a$ replaces solutions $u^\varepsilon$ by $u^\varepsilon + ta$, so we lose no generality in assuming that $H(x, 0) \leq -1$ for all $x \in \R^d$. It is well-known (see, e.g. Theorem 1.34 from~\cite{Tran-book}) that the solutions $u^\varepsilon$ are Lipschitz, with bound $\Lip(u^\varepsilon) \leq C$ independent of $\varepsilon$. In particular, only the values of $H(x, p)$ for $|p| \leq C$ are needed to solve the initial-value problem~(\ref{eps-cauchy}). Therefore, we lose no generality in assuming that $H(x, p) = |p|^2$ for $|p| \geq C$. We write $L(x, v)$ to denote the Lagrangian \[ L(x, v) := \sup_{p \in \R^d} p \cdot v - H(x, p), \] which we use to define the metric
\begin{equation}\label{eps-metric}
    m(t, x, y) := \inf_{\gamma \in \Gamma(t, x, y)} \int_0^t L(\gamma(s), \gamma'(s)) \; ds,
\end{equation}
where $\Gamma(t, x, y)$ is the set of paths $\gamma \in W^{1,1}([0, t]; \R^d)$ with $\gamma(0) = x$ and $\gamma(t) = y$. We also define the homogeneous metric
\begin{equation}\label{macro-metric}
    \overline{m}(t, x, y) := \lim_{n \to \infty} n^{-1}m(nt, nx, ny).
\end{equation}
Given a path $\gamma \in \Gamma(t, x, y)$, we refer to $\int_0^t L(\gamma(s), \gamma'(s)) \; ds$ as the cost of $\gamma$. Noting that the assumption on $H$ implies that $L(x, v) = |v|^2$ for $|v| \geq C$, it is a standard fact that a minimizer $\gamma \in \Gamma(t, x, y)$ exists for the infimum in equation~(\ref{eps-metric}) which satisfies
\begin{equation}\label{speed-limit}
    \Lip(\gamma) \leq C + Ct^{-1}|x-y|.
\end{equation}
The optimal control formulation of~(\ref{eps-cauchy}) is
\begin{equation}\label{eps-rep}
    u^\varepsilon(t, y) = \inf_{|x-y| \leq Ct} u_0(x) + \varepsilon m(\varepsilon^{-1}t, \varepsilon^{-1}x, \varepsilon^{-1}y).
\end{equation}
Define the cone $\Omega := \{(t, x) \in \R_{\geq 0} \times \R^d \mid |x| \leq Ct\}$. For any $(t, y-x) \in \Omega$, we have \[ |m(t, x, y) - m(\lceil t \rceil, [x], [y])| \leq C, \] where $[\cdot]$ denotes coordinate-wise rounding to integers in a way that stays inside $\Omega$. By $\Z^d$-periodicity, \[ m(\lceil t \rceil, [x], [y]) = m(\lceil t \rceil, 0, [y]-[x]). \] The Lipschitz estimate~(\ref{speed-limit}) for minimizers shows that $f(t, x) := m(t, 0, x)$ has approximate geodesics. Indeed, we find an approximate geodesic by chopping up a minimizing path, and the Lipschitz estimate~(\ref{speed-limit}) shows that the pieces lie in $\Omega$. Since $L(x, v) \geq 1$ for all $x, v \in \R^d$, it is clear that $f$ has linear growth (when restricted to $\Omega$) and is subadditive and nonnegative. We finish by applying Alexander's theorem, which yields \[ |\varepsilon m(\varepsilon^{-1}t, \varepsilon^{-1}x, \varepsilon^{-1}y) - \overline{m}(t, x, y) \leq C\varepsilon \log (C + \varepsilon^{-1}t + \varepsilon^{-1}|x-y|), \] for all $(t, x, y)$ with $(t, y-x) \in \Omega$. The main result follows.

\section{The case $d=2$}
In this section, we assume $d=2$. Rather than working with approximate geodesics as before, it will be more convenient to work directly with the minimizers for the metric $m$. We follow the same method as Alexander~\cite{Alexander90}, who proved an analogous result in the context of Bernoulli percolation. We first show that $m$ is approximately superadditive.
\begin{lem}\label{approx-additivity}
    If $(t, x) \in \Omega \cap \Z^{d+1}$, then $2m(t, 0, x) \leq m(2t, 2x) + C$.
\end{lem}
\begin{proof}
    Let $\gamma \in \Gamma(2t, 0, 2x)$ be a minimizing path.

    \textit{Step 1.} We show that we can form a path from $0$ to $x$ as the concatenation of at most $4$ non-overlapping segments of $\gamma$. Let $\gamma^1, \gamma^2 \colon [0, t] \to \R^d$ be the first and second halves of $\gamma$ respectively, given by \[ \gamma^1(s) := \gamma(s) \] and \[ \gamma^2(s) := \gamma(t+s)-\gamma(t) \] respectively. Then $\gamma^1(t) + \gamma^2(t) = 2x$, so $\gamma^1(t) = x - y$ and $\gamma^2(t) = x + y$ for some $y \in \R^d$. By a linear transformation of $\R \times \R^2$, we lose no generality in assuming that $x = 0$ and $y = (A, 0)$ for some $A > 0$. Consider the paths \[ \eta^1 \colon s \mapsto (s, \gamma^1(s) + (A, 0)) \] and \[ \eta^2 \colon s \mapsto (s, \gamma^2(s)), \]
    so $\eta^1(0) = (0, A, 0)$, $\eta^1(t) = (t, 0, 0)$, $\eta^2(0) = (0, 0, 0)$, and $\eta^2(t) = (t, A, 0)$.

    For $k \in \{1, 2\}$ and $c \in [0, t]$, we define the \textit{cyclic shift} \[ \eta^{k,c}(s) := \begin{cases} \eta^k(c+s)-\eta^k(c) &\quad \text{if $c+s \leq t$}\\ \eta^k(s-(t-c)) + \eta^k(t) - \eta^k(c) &\quad \text{otherwise.} \end{cases} \] We claim that some cyclic shifts of $\eta^1$ and $\eta^2$ intersect. Indeed, we can cyclically shift either path so that it is contained in the half-space $H^{\pm} := \{x \in \R^3 \mid \pm x \cdot (0, 0, 1) \geq 0\}$. The claim then follows from continuity, starting with $\eta^1$ in $H^-$ and $\eta^2$ in $H^+$, and cyclically shifting them into $H^+$ and $H^-$ respectively, as we will now explain in detail.

        Indeed, suppose that the cyclic shifts $\eta^{1,c_1}$ and $\eta^{2,c_2}$ do not intersect for any $c_1, c_2 \in [0, t]$. Then form the map $\varphi^{c_1, c_2} \colon [0, t] \to S^1$, where $S^1$ is the unit circle (identified in $\C = \R^2$ for concreteness) by \[ \varphi^{c_1, c_2}(s) := P\left(\frac{\eta^{1,c_1}(s)-\eta^{2,c_2}(s)}{|\eta^{1,c_1}(s)-\eta^{2,c_2}(s)|}\right), \] where $P(x, y, z) := (y, z)$ denotes projection onto the last two coordinates. Since the denominator is always nonzero, shifting $c_1$ and $c_2$ continuously produces a homotopy. As a homotopy invariant, the winding number of $\varphi^{c_1, c_2}$ is constant with respect to $c_1, c_2$. Choosing \[ c_1 := \argmax_c \eta^1(c) \cdot (0, 0, 1) \quad \text{and} \quad c_2 := \argmin_c \eta^2(c) \cdot (0, 0, 1) \] ensures $\eta^{1,c_1}(s) \in H^-$ and $\eta^{2,c_2}(s) \in H^+$ for all $s$. Since $(\eta^{1, c_1}(s) - \eta^{2, c_2}(s)) \cdot(0, 0, 1) \leq 0$, the map $\varphi^{c_1, c_2}$ is homotopic to $s \mapsto e^{-i\pi s/t}$, which has winding number $-1/2$. On the other hand, choosing \[ c_1 := \argmin_c \eta^1(c) \cdot (0, 0, 1) \quad \text{and} \quad c_2 := \argmax_c \eta^2(c) \cdot (0, 0, 1) \] makes $\varphi^{c_1, c_2}$ homotopic to $s \mapsto e^{i\pi s/t}$, which has winding number $1/2$, a contradiction.

        Finally, we form a new path following (a cyclic shift of) $\eta^2$ from $(0, 0, 0)$ to the point of intersection, and following $\eta^1$ the rest of the way to $(t, 0, 0)$.

        To summarize, we found a path from $0$ to $x$ which is composed of a segment of a cyclic shift of $\gamma^1$ and a segment of a cyclic shift of $\gamma^2$, so the segments don't overlap. Since we took cyclic shifts, this equates to at most $4$ segments from $\gamma$.

        \textit{Step 2.} Use Step 1 to find an approximate geodesic with subsequence \[ 0 = (t_0, x_0), (t_1, x_1), \dots, (t_9, x_9) = (2t, 2x) \] for $m$ along $\gamma$, such that there are indices $i_1, \dots, i_4$ with $\sum_{k=1}^4 (t_{i_k}-t_{i_k-1}, x_{i_k}-x_{i_k-1}) = (t, x)$. Rearranging the indices, we find a path $\tilde{\gamma} \in \Gamma(2t, 0, 2x)$ with $\tilde{\gamma}(t) = x$ and cost at most $C$ more than the cost of $\gamma$. The conclusion follows.
\end{proof}

The previous lemma and subadditivity show that \[ m(2t, 0, 2x) \leq 2m(t, 0, x) \leq m(2t, 0, 2x) + C \] for all $(t, x) \in \Omega \cap \Z^{d+1}$. Then $m(t, 0, x) - C \leq 2^{-k}(m(2^k t, 0, 2^k x) - C)$ for all $k \in \N$ by induction, so letting $k \to \infty$ shows \[ |m(t, x, y) - \overline{m}(t, x, y)| \leq C \] for all $(t, x) \in \Omega \cap \Z^{d+1}$, so the same holds for all $(t, x) \in \Omega$ since $m$ is Lipschitz. The result in dimension $d=2$ follows.

\section*{Acknowledgement} I would like to thank my advisor, Charles Smart, for many helpful discussions and comments on earlier drafts of this paper.
\printbibliography{}
\end{document}